\documentclass[reqno,a4paper,12pt]{amsart}

\usepackage[all,poly]{xy}
\usepackage{amsfonts,stmaryrd}
\usepackage[mathcal]{eucal}
\usepackage{amssymb}
\usepackage{amsmath}
\usepackage{mathrsfs}
\usepackage{enumerate}

\def\K{\operatorname{K}}

\def\map{\longrightarrow}

\def\GL{\operatorname{GL}}

\def\SL{\operatorname{SL}}

\def\Sp{\operatorname{Sp}}

\def\St{\operatorname{St}}

\def\EE{\operatorname{EE}}
\def\EU{\operatorname{EU}}
\def\SK{\operatorname{SK}}

\def\sr{\operatorname{sr}}

\def\Rat{{\mathbb Q}}

\def\Int{{\mathbb Z}}

\def\prim{\mathfrak p}

\def\pamod#1{\,(\operatorname{mod}{\, #1})\,}

\newtheorem{theorem}{Theorem}

\newtheorem{lemma}{Lemma}
\newtheorem{problem}{Problem}

\title[Multiple commutators of
elementary subgroups]{multiple commutators of
elementary subgroups: end of the line}
\author{Nikolai Vavilov}
\address{St. Petersburg State University}
\email{nikolai-vavilov@yandex.ru}
\thanks{This publication is supported by Russian Science Foundation grant 17-11-01261.}
\author{Zuhong Zhang}
\address{Department of  Mathematics, Beijing Institute
of Technology, Beijing, China}
\email{zuhong@hotmail.com}
 \date{}

\begin{document}


\begin{abstract}
In our previous joint papers with Roozbeh Hazrat and Alexei
Stepanov we established commutator formulas for relative
elementary subgroups in $\GL(n,R)$, $n\ge 3$, and other 
similar groups, such as Bak's unitary groups, or Chevalley 
groups. In particular, there it was shown that
multiple commutators of elementary subgroups can be reduced to
double such commutators. However, since the proofs of these 
results depended on the standard commutator formulas, it 
was assumed that the ground ring $R$ is quasi-finite. Here we
propose a different approach which allows to lift any such 
assumptions and establish almost definitive results. In particular,
we prove multiple commutator formulas, and other related 
facts for $\GL(n,R)$ over an {\it arbitrary} associative ring $R$.
\end{abstract}

\maketitle


\section{Introduction}
 	
In the present paper we put the last dot in the proof of the multiple commutator formula for relative/unrelative elementary subgroups of
$\GL(n,R)$. Namely, we establish that any higher such commutator
subgroup is equal to a double commutator subgroup of this form.
\par
More precisely, we generalise \cite{yoga-2}, Theorem 8A, and \cite{Hazrat_Vavilov_Zhang}, Theorem 5A, from quasi-finite rings, 
to arbitrary associative rings. The main result of the present paper 
asserts that for an {\it arbitrary\/} associative ring $R$, its arbitrary
 two-sided ideals $I_i\unlhd R$,  $i=1,\ldots,m$, and an arbitrary arrangment of brackets $[\![\ldots]\!]$ with the cut point\/ $s$
in the multiple commutator, one has 
$$ [\![E(n,I_1),E(n,I_2),\ldots,E(n,I_m)]\!]=
[E(n,I_1\circ\ldots\circ I_s),E(n,I_{s+1}\circ\ldots\circ I_m)], $$
\noindent
where $A\circ B=AB+BA$ is the symmetrised product of ideals,
see \S~4 for the precise statement. Most of the intermediate 
results are established for $n\ge 3$, but at one key point at the 
very end of the proof we have to assume that $n\ge 4$. In turn,
it is classically known that double such commutators are not in 
general equal to a single elementary subgroup, see \S~8.
\par
Jointly with Roozbeh Hazrat and Alexei Stepanov, we published
several similar results before, see, for instance, \cite{Hazrat_Zhang_multiple,
yoga-2, RNZ5, Stepanov_universal, Hazrat_Vavilov_Zhang}, not
just for the general linear group, but also for other types of 
groups, including Bak's unitary groups and Chevalley groups. 
However, all these results vitally depended on some additional 
assumptions on the ground rings, either some form of 
commutativity conditions, or some form of stability/dimension conditions.
\par
The reason was that the proofs of these results always hinged
on level calculations, with subsequent invocation of some form 
of the relative standard commutator formula. In \cite{yoga-2} we 
even say: ``In the next section we embark on the [somewhat 
easier] calculation of higher commutators of relative elementary subgroups $\ldots$ Even this turns out to be a rather non-trivial 
task. In fact, we do not see any other way to do that, but to 
prove a higher analogue of the standard commutator formula.''
\par
For $\GL(n,R)$ such birelative formulas at the stable level are 
classically known from the work of Hyman Bass, Alec Mason 
and Wilson Stothers \cite{Bass_stable,Mason_Stothers}. For 
rings subject to commutativity conditions, such formulas were 
established by Hong You, and then in our joint papers with 
Roozbeh Hazrat and Alexei Stepanov \cite{Vavilov_Stepanov_standard, Hazrat_Zhang, Vavilov_Stepanov_revisited, yoga-1, Hazrat_Vavilov_Zhang},
which in turn relied on a breed of the powerful methods proposed
by Andrei Suslin, Zenon Borewicz and ourselves, Leonid Vaserstein, 
Tony Bak, and others to prove standard commutator formulas in 
the absolute case, see \cite{Suslin, borvav, Vaserstein_normal, Stepanov_Vavilov_decomposition, Bak}.
\par
However, from the work of Victor Gerasimov \cite{Gerasimov}
it is known that the standard commutator formula may fail for 
general associative rings, even in the absolute case. Thus, there
is no hope whatsoever to establish the multiple commutator
formula over arbitrary commutative rings by the methods used
in \cite{Hazrat_Zhang_multiple,
yoga-2, RNZ5, Stepanov_universal, Hazrat_Vavilov_Zhang}.
\par
In the present paper we launch a completely different approach,
which entirely relies on {\it elementary\/} calculations inside
the absolute elementary group $E(n,R)$ itself. The starting point
of this approach were our recent papers on unrelativisation, and
restrained generating sets of elementary commutators, see
\cite{NV18, NZ1, NZ2, NZ3}. In fact, most of the calculations in the present paper are enhancement, refinement, or spin-off of our calculations in \cite{NZ1, NZ2, NZ3}. Of course, because of 
non-commutativity now we have to be much more circumspective 
and meticulous on several occasions.

Morally, these calculations are just a new stage of the vintage
elementary calculations in lower $K$-theory, as developed since 
1960-ies. We were
especially influenced by the works of Wilberd van der Kallen, 
such as \cite{vdK-group}. Since our calculations only depend on
Steinberg relations, they carry over {\it verbatim\/} also to 
Steinberg groups. In fact, a recent preprint by Andrei Lavrenov
and Sergei Sinchuk \cite{Lavrenov_Sinchuk} contained calculations
in the same spirit, over a commutative ring $R$.
\par
The paper is organised as follows. In \S~2 we recall some notation 
pertaining to multiple commutators, and in \S~3 briefly review 
the requisite facts on elementary subgroups in $\GL(n,R)$.
After that in \S~4 we state the main result of the present paper, 
the multiple elementary commutator formula, and show that it
can be easily derived from the two special cases: triple
commutators, and quadruple commutators. The proof for these 
special cases occupies \S\S~5--6, and it is our main new contribution,
and the technical core of the whole paper. Finally, in \S~7 we make some further related observations,
including a generalisation, from quasi-finite to arbitrary associative 
rings, of another theorem by the first author and Stepanov on 
$[E(n,A),E(n,B)]$ for coprime ideals $A+B=R$, and state some 
unsolved problems.


\section{Multiple commutators}

Let $G$ be a group. A subgroup $H\le G$ generated by a
subset $X\subseteq G$ will be denoted by $H=\langle X\rangle$.
For two elements $x,y\in G$ we denote by ${}^xy=xyx^{-1}$
and $y^x=x^{-1}yx$ the left and right conjugates of $y$ by $x$,
respectively. Further, we denote by 
$$ [x,y] = xyx^{-1}y^{-1}={}^xy\cdot y^{-1}=x\cdot{}^y{x^{-1}} $$
\noindent
the left-normed commutator of $x$ and $y$. Our multiple 
commutators are also left-normed. Thus, by default, $[x,y,z]$
denotes $[[x,y],z]$, we will use different notation for other
arrangement of brackets. Throughout the present paper we
repeatedly use the customary commutator identities, such as 
their multiplicativity with respect to the factors:
$$ [x,yz]=[x,y]\cdot{}^y[x,z],\qquad [xy,z]={}^x[y,z]\cdot[x,z], $$
\noindent
and a number of other similar identities, such as 
$${ [x,y]}^{-1}=[y,x],\qquad {}^z[x,y]=[{}^zx,{}^zy],\qquad 
[x^{-1},y]=[y,x]^x,\qquad [x,y^{-1}]=[y,x]^y, $$
\noindent
usually without any specific reference.
Iterating multiplicativity we see that the commutator 
$[x_1\ldots x_m,y]$ is the product of conjugates of the 
commutators $[x_i,y]$, $i=1,\ldots,m$. Obviously, a similar
claim holds also for $[x,y_1\ldots y_m]$.
\par
Further, for two subgroups $F,H\le G$ one denotes by $[F,H]$ 
their mutual commutator subgroup, spanned by all commutators 
$[f,h]$, where $f\in F$, $h\in H$. Clearly, $[F,H]=[H,F]$, and 
if $F,H\unlhd G$ are normal in $G$, then $[F,H]\unlhd G$ is also 
normal.
In the main results of the present paper, we wish to establish that
$[F,H]\le K$, for two subgroups $F,H\in G$, and a normal subgroup
$N\unlhd G$. This will be done as follows. 

\begin{lemma}
Let $G$ be a group, $F,H\le G$ be its subgroups and $N\unlhd G$
be its normal subgroup. Further, let $F=\langle X\rangle$ and
$H=\langle Y\rangle$ for some $X,Y\subseteq G$. Then
$$ [F,H]\le N\qquad \Longleftrightarrow\qquad 
[x,y]\in N,\quad x\in X,\ y\in Y. $$
\end{lemma}

To state our main results, we have to recall some further 
pieces of notation from \cite{yoga-1,Hazrat_Zhang_multiple,yoga-2, RNZ5, Hazrat_Vavilov_Zhang, Stepanov_universal}.
Namely, let $H_1,\ldots,H_m\le G$ be subgroups of $G$. There are 
many ways to form a higher commutator of these
groups, depending on where we put the brackets. Thus, for three
subgroups $F,H,K\le G$ one can form two triple commutators
$[[F,H],K]$ and $[F,[H,K]]$. Usually, we write $[H_1,H_2,\ldots,H_m]$ for the {\it left-normed\/} commutator, defined inductively by
$$ [H_1,\ldots,H_{m-1},H_m]=[[H_1,\ldots,H_{m-1}],H_m]. $$
\noindent
To stress that here we consider {\it any\/} commutator of these subgroups, with an arbitrary placement of brackets, we write $[\![H_1,H_2,\ldots,H_m]\!]$. Thus, for instance, $[\![F,H,K]\!]$ 
refers to any of the two arrangements above.
\par
Actually, a specific arrangment of brackets usually does not play
major role in our results -- apart from one important 
attribute\footnote{Actually, for non-commutative rings 
symmetric product of ideals is not associative, so that 
the initial bracketing of higher commutators will be reflected also 
in the bracketing of such higher symmetric products.}. 
Namely, what will matter a lot is the position of the outermost 
pairs of inner brackets. Namely, every higher commutator subgroup
$[\![H_1,H_2,\ldots,H_m]\!]$ can be uniquely written as
$$ [\![H_1,H_2,\ldots,H_m]\!]=
[[\![H_1,\ldots,H_s]\!],[\![H_{s+1},\ldots,H_m]\!]], $$
\noindent
for some $s=1,\ldots,m-1$. This $s$ will be called the cut point
of our multiple commutator. 


\section{Relative subgroups}

Let $G=\GL(n,R)$ be the general linear group of degree $n$ over an
associative ring $R$ with 1. In the sequel for a matrix
$g\in G$ we denote by $g_{ij}$ its matrix entry in the position
$(i,j)$, so that $g=(g_{ij})$, $1\le i,j\le n$. The inverse of
$g$ will be denoted by $g^{-1}=(g'_{ij})$, $1\le i,j\le n$.
\par
As usual we denote by $e$ the identity matrix of degree $n$
and by $e_{ij}$ a standard matrix unit, i.~e., the matrix
that has 1 in the position $(i,j)$ and zeros elsewhere.
An elementary transvection $t_{ij}(\xi)$ is a matrix of the
form $t_{ij}(c)=e+c e_{ij}$, $1\le i\neq j\le n$, $c\in R$. 
\par
Further, let $A$ be a two-sided of $R$. We consider the
corresponding reduction homomorphism
$$ \pi_A:\GL(n,R)\map\GL(n,R/A),\quad
(g_{ij})\mapsto(g_{ij}+A). $$
\noindent
Now, the {\it principal congruence subgroup} $\GL(n,R,A)$ of level $A$
is the kernel $\pi_A$, 
\par
The {\it unrelative elementary subgroup} $E(n,A)$ of level $A$
in $\GL(n,R)$ is generated by all elementary matrices of level 
$A$. In other words,
$$ E(n,A)=\langle e_{ij}(a),\ 1\le i\neq j\le n,\ a\in A \rangle. $$
\noindent
In general $E(n,A)$ has little chances to be normal in $\GL(n,R)$. 
The {\it relative elementary subgroup} $E(n,R,A)$  of level $A$
is defined as the normal closure of $E(n,A)$ in the absolute 
elementary subgroup $E(n,R)$:
$$ E(n,R,A)=\langle e_{ij}(a),\ 1\le i\neq j\le n,\ a\in A 
\rangle^{E(n,R)}. $$

The following lemma in generation of relative elementary subgroups
$E(n,R,A)$ is a classical result discovered in various contexts by 
Stein, Tits and Vaserstein, see, for instance, \cite{Vaserstein_normal}
(or \cite{Hazrat_Vavilov_Zhang}, Lemma 3, for a complete elementary
proof). It is stated in terms of the {\it Stein---Tits---Vaserstein 
generators\/}):
$$ z_{ij}(a,c)=t_{ij}(c)t_{ji}(a)t_{ij}(-c),\qquad
1\le i\neq j\le n,\quad a\in A,\quad c\in R. $$

\begin{lemma}
Let $R$ be an associative ring with $1$, $n\ge 3$, and let $A$ 
be a two-sided ideal of $R$. Then as a subgroup $E(n,R,A)$ is 
generated by $z_{ij}(a,c)$, for all $1\le i\neq j\le n$, $a\in A$, 
$c\in R$.
\end{lemma}

Our main instrument in this paper is the following generalisation
of Lemma~1 to mutual commutator subgroups 
$[E(n,R,A),E(n,R,B)]$ of relative elementary subgroups.
\par
Denote by $A\circ B=AB+BA$ the symmetrised product of two-sided ideals $A$ and $B$. For commutative rings, $A\circ B=AB=BA$
is the usual product of ideals $A$ and $B$. However, in general, 
the symmetrised product is not associative. Thus, when writing something like $A\circ B\circ C$, we have to specify the order 
in which products are formed. The following level computation is
standard, see, for instance, \cite{Vavilov_Stepanov_standard,
Vavilov_Stepanov_revisited, Hazrat_Vavilov_Zhang}, and references
there.

\begin{lemma}
$R$ be an associative ring with $1$, $n\ge 3$, and let $A$ and $B$
be two-sided ideals of $R$.  Then 
$$ E(n,R,A\circ B)\le\big[E(n,A),E(n,B)\big]\le
\big[E(n,R,A),E(n,R,B)\big] \le\GL(n,R,A\circ B). $$
\end{lemma}

In the following theorem a further type of generators occur, the
{\it elementary commutators\/}:
$$ y_{ij}(a,b)=[t_{ij}(a),t_{ji}(b)],\qquad
1\le i\neq j\le n,\quad a\in A,\quad b\in B. $$

The following analogue of Lemma~1 for commutators 
$[E(n,R,A),E(n,R,B)]$ was discovered (in slightly less precise 
forms) by Roozbeh Hazrat and the second author, see \cite{Hazrat_Zhang_multiple}, Lemma 12 and then in our joint 
paper with Hazrat \cite{Hazrat_Vavilov_Zhang}, Theorem~3A. 
The strong form reproduced below is established only in our
paper \cite{NZ2}, Theorem~1, as an aftermath of our papers
\cite{NV18, NZ1}.

\begin{lemma}
Let $R$ be any associative ring with $1$, let $n\ge 3$, and let $A,B$ 
be two-sided ideals of $R$. Then the mixed commutator subgroup 
$[E(n,R,A),E(n,R,B)]$ is generated as a group by the elements of the form
\par\smallskip
$\bullet$ $z_{ij}(ab,c)$ and $z_{ij}(ba,c)$,
\par\smallskip
$\bullet$ $y_{ij}(a,b)$,
\par\smallskip\noindent
where $1\le i\neq j\le n$, $a\in A$, $b\in B$, $c\in R$. Moreover,
for the second type of generators, it suffices to fix one pair of
indices $(i,j)$.
\end{lemma}

Since all generators listed in Lemma~3 belong already to the
commutator subgroup of unrelative elementary subgroups, we
get the following corollary, \cite{NZ2}, Theorem 2.

\begin{lemma}
Let $R$ be any associative ring with $1$, let $n\ge 3$, and let $A,B$ 
be two-sided ideals of $R$.  Then one has
$$ \big[E(n,R,A),E(n,R,B)\big]=\big[E(n,R,A),E(n,B)\big]=\big[E(n,A),E(n,B)\big]. $$
\end{lemma}

In particular, it follows that $[E(n,A),E(n,B)]$ is normal in
$E(n,R)$. In fact, \cite{NZ2}, Lemma~3, implies a much stronger
fact that the quotient
$$ \big[E(n,R,A),E(n,R,B)\big]/E(n,R,A\circ B)=
\big[E(n,A),E(n,B)\big]/E(n,R,A\circ B) $$ 
\noindent
is central in $E(n,R)/E(n,R,A\circ B)$. In other words, 

\begin{lemma}
Let $R$ be an associative ring with $1$, $n\ge 3$, and let $A,B$ 
be two-sided ideals of $R$. Then
$$ \big[\big[E(n,A),E(n,B)\big],E(n,R)\big]=E(n,R,AB+BA). $$
\end{lemma}

In particular, $\big[E(n,A),E(n,B)\big]/E(n,R,A\circ B)$ is
itself abelian.


\section{Main theorem}

In this section we consider applications of \cite{NZ2}, Theorem 1 
(= Lemma 4 above) to multiple commutators. In particular, here 
we generalise \cite{yoga-2}, Theorem 8A, and \cite{Hazrat_Vavilov_Zhang}, 
Theorem 5A, from quasi-finite rings, to arbitrary associative rings. 
Namely, we prove that any multiple commutator of relative or
unrelative elementary subgroups is  equal to some double such
commutator. 

\begin{theorem}
Let $R$ be any associative ring with $1$, let $n\ge 4$, and let 
$I_i\unlhd R$, $i=1,\ldots,m$,  be two-sided ideals of $R$. 
Consider an arbitrary arrangment of brackets\/ $[\![\ldots]\!]$
with the cut point\/ $s$. Then one has
$$
[\![E(n,I_1),E(n,I_2),\ldots,E(n,I_m)]\!]=
[E(n,I_1\circ\ldots\circ I_s),E(n,I_{s+1}\circ\ldots\circ I_m)], $$
\noindent
where the bracketing of symmetrised products on the right hand side coincides with the bracketing of the commutators on the left hand side.
\end{theorem}

Actually this theorem easily follows by induction on $m$ from
the following two special cases, triple commutators, and quadruple commutators.

\begin{lemma}
Let $R$ be an associative ring with $1$, $n\ge 3$, and let $A,B,C$ 
be two-sided ideals of $R$. Then
$$ \big[\big[E(n,A),E(n,B)\big],E(n,C)\big]=
\big[E(n,A\circ B),E(n,C)\big]. $$
\end{lemma}

\begin{lemma}
Let $R$ be an associative ring with $1$, $n\ge 4$, and let $A,B,C,D$ 
be two-sided ideals of $R$. Then
$$ \big[\big[E(n,A),E(n,B)\big],\big[E(n,C),E(n,D)\big]\big]=
\big[E(n,A\circ B),E(n,C\circ D)\big]. $$
\end{lemma}

The proofs of Lemma 3 is rather tricky (and a proof of Lemma 8 for
$n=3$ would be even much trickier than that), and will be postponed 
to the following sections, which constitute the technical core of the 
paper. However, it is very easy to see that Theorem 1 immediately follows.

\begin{proof}
Denote the commutator on the left-hand side by $H$,
$$ H=[\![E(n,I_1),E(n,I_2),\ldots,E(n,I_m)]\!]. $$
\noindent
We argue by induction in $m$, with the cases $m\le 4$ as the
base of induction --- for the case $m=2$ there is nothing to
prove, case $m=3$ is accounted for by Lemma 7, and case
$m=4$ --- by Lemma 7, if the cut point $s\neq 2$, and by 
Lemma 8 when $s=2$.
\par
Now, let $m\ge 5$ and assume that our theorem is already 
proven for all shorter commutators. Consider an arbitrary 
arrangment of brackets\/ $[\![\ldots]\!]$ with the cut point 
$s$ and let 
$$ [\![E(n,I_1),E(n,I_2),\ldots,E(n,I_s)]\!],\qquad
[\![E(n,I_{s+1}),E(n,I_{s+2}),\ldots,E(n,I_m)]\!], $$
\noindent
be the partial commutators, the first one containing the factors 
afore the cut point, and the second one containing those after
the cut point. 
\par\smallskip
$\bullet$ When the cut point occurs at $s=1$ or at $s=m-1$, one 
of these commutators is a single elementary subgroup $E(n,I_1)$ 
in the first case or $E(n,I_{m-1})$ in the second one.  Then we 
can apply the induction hypothesis to another factor. For $s=1$,
denote by $t=2,\ldots,m-1$ the cut point of the second factor.
Then by induction hypothesis
\begin{multline*}
H=[E(n,I_1),[\![E(n,I_2),E(n,I_3),\ldots,E(n,I_m)]\!]]=\\
[E(n,I_1),[E(n,I_1\circ\ldots\circ I_t),
E(n,I_{t+1}\circ\ldots\circ I_m)]], 
\end{multline*}
\noindent
and we are done by Lemma 7. Similarly, for $s=m-1$ denote by 
$r=1,\ldots,m-1$ the cut point of the first factor. Then by 
induction hypothesis
\begin{multline*}
H=[[\![E(n,I_1),E(n,I_2),\ldots,E(n,I_{m-1})]\!],E(n,I_m)]=\\
[[E(n,I_1\circ\ldots\circ I_r),
E(n,I_{r+1}\circ\ldots\circ I_{m-1}),E(n,I_m)], 
\end{multline*}
\noindent
and we are again done by Lemma 7. 
\par\smallskip
$\bullet$ Otherwise, when $s\neq 1,m-1$, we can apply the 
induction hypothesis to both factors. Let as above $r=1,\ldots,s-1$
be the cut point of the first factor and let $t=s+1,\ldots,m-1$ be
the cut point of the second factor. Then we can apply induction
hypothesis to both factors of
$$ H=[[\![E(n,I_1),E(n,I_2),\ldots,E(n,I_s)]\!],
[\![E(n,I_{s+1}),E(n,I_{s+2}),\ldots,E(n,I_m)]\!]] $$
\noindent
to conclude that
$$ H=[[E(n,I_1\circ\ldots\circ I_r),E(n,I_{r+1}\circ\ldots\circ I_s)],
[E(n,I_{s+1}\circ\ldots\circ I_t),E(n,I_{t+1}\circ\ldots\circ I_m)]], $$
\noindent
and we are again done, this time by Lemma 8.
\end{proof}


\section{Elementary commutators modulo $E(n,R,A\circ B)$}

Our proofs of Lemma 7 and Lemma 8 are refinements of the
calculations inside the proof of \cite{NZ2}, Theorem 1, namely
the proofs of Lemmas 3 and 5 thereof (Lemma 4 establishes
that the third type of generators used in \cite{Hazrat_Zhang_multiple,
Hazrat_Vavilov_Zhang} are redundant). For convenience of the 
reader, we reproduce these lemmas in our notation, and in 
somewhat more precise form we use throughout the present 
paper\footnote{Besides, we need these more general congruences 
also to compute some double elementary commutator subgroups
in \S~8.}.

\begin{lemma}
Let $R$ be an associative ring with $1$, $n\ge 3$, and let $A,B$ 
be two-sided ideals of $R$. Then for any  $1\le i\neq j\le n$, 
$a\in A$, $b\in B$, and any $x\in E(n,R)$ one has
$$ {}^x y_{ij}(a,b)\equiv  y_{ij}(a,b) \pamod{E(n,R,A\circ B)}. $$
\end{lemma}
\begin{proof}
Clearly, $y_{ij}(a,b)$ resides in the image of
the fundamental embedding of $E(2,R)$ into $E(n,R)$ in the
$i$-th and $j$-th rows and columns, where one has
\par\smallskip
$$ y_{ij}(a,b)=\left[\begin{pmatrix} 1&a\\ 0&1\end{pmatrix},
\begin{pmatrix} 1&0\\ b&1\end{pmatrix}\right]=
\begin{pmatrix} 1+ab+abab&-aba\\ bab&1-ba\end{pmatrix} $$
\noindent
and 
$$ y_{ij}(a,b)^{-1}=
\left[\begin{pmatrix} 1&0\\ b&1\end{pmatrix},
\begin{pmatrix} 1&a\\ 0&1\end{pmatrix}\right]=
\begin{pmatrix} 1-ab&aba\\ -bab&1+ba+baba\end{pmatrix}. $$
\par\smallskip
Consider the elementary conjugate ${}^xy_{ij}(a,b)$. We argue by induction on the length of $x\in E(n,R)$ in elementary generators. 
Let $x=yt_{kl}(c)$, where $y\in E(n,R)$ is shorter than $x$, whereas
$1\le k\neq l\le n$, $c\in R$. 
\par\smallskip
$\bullet$ If $k,l\neq i,j$, then $t_{kl}(c)$ commutes with $y_{ij}(a,b)$ 
and can be discarded.
\par\smallskip
$\bullet$ On the other hand, for any $h\neq i,j$ the above formulas 
for $y_{ij}(a,b)$ and $y_{ij}(a,b)^{-1}$ immediately imply that
\begin{align*}
&[t_{ih}(c),y_{ij}(a,b)]=t_{ih}(-abc-ababc)t_{jh}(-babc),\\
&[t_{jh}(c),y_{ij}(a,b)]=t_{ih}(abac)t_{jh}(bac),\\
&[t_{hi}(c),y_{ij}(a,b)]=t_{hi}(cab)t_{hj}(-caba),\\
&[t_{hj}(c),y_{ij}(a,b)]=t_{hi}(cbab)t_{hj}(-cba-cbaba).
\end{align*}
\par\noindent
All factors on the right hand side belong already to $E(n,A\circ B)$
This means that
$$ {}^xy_{ij}(a,b)\equiv {}^yy_{ij}(a,b) \pamod{E(n,R,A\circ B)}. $$
\par\smallskip
$\bullet$ Finally, for $(k,l)=(i,j),(j,i)$ we can take an $h\neq i,j$
and rewrite $t_{kl}(c)$ as a commutator $t_{ij}(c)=[t_{ih}(c),t_{hj}(1)]$
or $t_{ji}(c)=[t_{h}(c),t_{hi}(1)]$ and apply the previous item to
get the same congruence modulo $E(n,R,A\circ B)$.
\par\smallskip
By induction we get that ${}^xy_{ij}(a,b)\equiv y_{ij}(a,b)\pamod{E(n,R,A\circ B)}$, as claimed.
\end{proof}

Lemma 9 immediately implies the following additivity property of
the elementary commutators with respect to its arguments.

\begin{lemma}
Let $R$ be an associative ring with $1$, $n\ge 3$, and let $A,B$ 
be two-sided ideals of $R$. Then for any  $1\le i\neq j\le n$, 
$a,a_1,a_2\in A$, $b,b_1,b_2\in B$ one has
\begin{align*}
&y_{ij}(a_1+a_2,b)\equiv  y_{ij}(a_1,b)\cdot y_{ij}(a_1,b) 
\pamod{E(n,R,A\circ B)},\\
&y_{ij}(a,b_1+b_2)\equiv  y_{ij}(a,b_1)\cdot y_{ij}(a,b_2) 
\pamod{E(n,R,A\circ B)},\\
&y_{ij}(a,b)^{-1}\equiv  y_{ij}(-a,b)\equiv y_{ij}(a,-b) 
\pamod{E(n,R,A\circ B)},\\
&y_{ij}(ab_1,b_2)\equiv y_{ij}(a_1,a_2b)\equiv e
\pamod{E(n,R,A\circ B)}.
\end{align*}
\end{lemma}
\begin{proof}
The first item can be derived from Lemma~9 as follows. By
definition
$$ y_{ij}(a_1+a_2,b)=[t_{ij}(a_1+a_2),t_{ji}(b)]=
[t_{ij}(a_1)t_{ij}(a_2),t_{ji}(b)], $$
\noindent
and it only remains to apply multiplicativity of commutators in 
the first factor, and then Lemma 9. The second item is similar, 
and the third item follows. The last item is obvious from the 
definition.
\end{proof}

The following result is a stronger, and more precise version of
\cite{NZ2}, Lemma 5.

\begin{lemma}
Let $R$ be an associative ring with $1$, $n\ge 3$, and let $A,B$ 
be two-sided ideals of $R$. Then for any  $1\le i\neq j\le n$, any
$1\le k\neq l\le n$, and all $a\in A$, $b\in B$, $c\in R$, one has
$$ y_{ij}(ac,b)\equiv  y_{kl}(a,cb) \pamod{E(n,R,A\circ B)}. $$
\end{lemma}
\begin{proof}
First, we show that we can move the second index of an
elementary commutator. With this end, we take any $ h\neq i,j$ 
and rewrite the elementary commutator
$ y_{ij}(ac,b)=\big[t_{ij}(ac),t_{ji}(b)\big]$ as
$$  y_{ij}(ac,b)=t_{ij}(ac)\cdot{}^{t_{ji}(b)}t_{ij}(-ac)=
t_{ij}(a)\cdot{}^{t_{ji}(b)}\big[t_{ih}(a),t_{hj}(-c)\big]. $$
\noindent
Expanding the conjugation by $t_{ji}(b)$, we see that 
$$  y_{ij}(ac,b)=t_{ij}(ac)\cdot\big[{}^{t_{ji}(b)}t_{ih}(a),{}^{t_{ji}(b)}t_{hj}(-c)\big] = 
t_{ij}(ac)\cdot[t_{jh}(ba)t_{ih}(a),t_{hj}(-c)t_{hi}(cb)\big]. $$
\noindent
Now, the first factor $t_{jh}(ba)$ of the first argument in this last commutator already belongs to the group $E(n,BA)$ which is 
contained in $E(n,R,A\circ B)$. Thus, as above, 
$$  y_{ij}(ac,b)\equiv  t_{ij}(a)\cdot[t_{ih}(a),t_{hj}(-c)t_{hi}(cb)\big] \pamod{E(n,R,A\circ B)}. $$
\noindent
Using multiplicativity of the commutator w.r.t. the second argument, cancelling the first two factors of the resulting 
expression, and then applying Lemma~9 we see that
$$  y_{ij}(ac,b)\equiv 
{}^{t_{hj}(-c)}\big[t_{ih}(a),t_{hi}(cb)\big] 
\equiv \big[t_{ih}(a),t_{hi}(cb)\big]\equiv y_{ih}(a,cb)
\pamod{E(n,R,A\circ B)}. $$
\par
Similarly, we can move the first index of an elementary commutator 
by rewriting the commutator $ y_{ij}(a,cb)$ differently, as
$$  y_{ij}(a,cb)=\big[t_{ij}(a),t_{ji}(cb)\big]=
{}^{t_{ij}(a)}t_{ji}(cb)\cdot t_{ji}(-cb)=
{}^{t_{ij}(a)}\big[t_{jh}(c),t_{hi}(b)\big]\cdot t_{ji}(-cb), $$
\noindent
we get the congruence 
$$  y_{ij}(a,cb)\equiv y_{hj}(ac,b) 
\pamod{E(n,R,A\circ B)}. $$
\par
Obviously, for $n\ge 3$ we can pass from any position $(i,j)$, 
$i\neq j$, to any other such position $(k,l)$, $k\neq l$, by a 
sequence of at most three such elementary moves, simultaneously 
moving $c$ from the first factor to the second one, or vice versa.
\end{proof}

Together with the previous results this lemma immediately
implies the following corollary, asserting that the elementary
commutators with parameters from $A^2$ or $B^2$ really
become elementary modulo $E(n,R,A\circ B)$.

\begin{lemma}
Let $R$ be an associative ring with $1$, $n\ge 3$, and let $A,B$ 
be two-sided ideals of $R$. Then for any  $1\le i\neq j\le n$, 
$a,a_1,a_2\in A$, $b,b_1,b_2\in B$ one has
$$ y_{ij}(a_1a_1,b)\equiv y_{ij}(a,b_1b_2)\equiv e
\pamod{E(n,R,A\circ B)}. $$
\end{lemma}
\begin{proof}
Indeed, by Lemmas 11  and 10 one has
\begin{align*}
&y_{ij}(a_1a_2,b)\equiv y_{ij}(a_1,a_2b)\equiv e
\pamod{E(n,R,A\circ B)},\\
&y_{ij}(a,b_1b_2)\equiv y_{ij}(ab_1,b_2)\equiv e
\pamod{E(n,R,A\circ B)}.
\end{align*}
\end{proof}


\section{Triple and quadruple commutators}

In this section we prove Lemma 7 and Lemma 8. We start with
Lemma 7.
\par\smallskip
First of all, let $1\le i\neq j\le n$, $1\le h\neq k\le n$, $a\in A$, 
$b\in B$, $c\in C$, $d\in R$ and $x\in E(n,R)$. Then, clearly,
$$ [{}^xz_{ij}(ab,d),t_{hk}(c)],  [{}^xz_{ij}(ba,d),t_{hk}(c)]\in 
[E(n,R,A\circ B),E(n,C)]. $$
\noindent
On the other hand, ${}^xy_{ij}(a,b)=y_{ij}(a,b)z$, for some 
$z\in E(n,R,A\circ B)$, and thus
$$ \big[{}^xy_{ij}(a,b),t_{hk}(c)\big]=
\big[y_{ij}(a,b)z,t_{hk}(c)\big]={}^{y_{ij}(a,b)}[z,t_{hk}(c)]\cdot
[y_{ij}(a,b),t_{hk}(c)]. $$ 
\noindent
The first of these commutators also belongs to
$[E(n,R,A\circ B),E(n,C)]$, and stays there after elementary
conjugations. Let's concentrate at the second one. The same 
analysis as in the proof of \cite{NZ2}, Lemma~3, shows that
\par\smallskip
$\bullet$ If $k,l\neq i,j$, then $t_{kl}(c)$ commutes with 
$y_{ij}(a,b)$.
\par\smallskip
$\bullet$ On the other hand, for any $h\neq i,j$ the formulas 
for $y_{ij}(a,b)$ and $y_{ij}(a,b)^{-1}$ given in the proof of 
Lemma~3 immediately imply that
\begin{align*}
&[y_{ij}(a,b),t_{ih}(c)]=t_{ih}(abc+ababc)t_{jh}(babc)\in E(n,ABC),\\
&[y_{ij}(a,b),t_{jh}(c)]=t_{ih}(-abac)t_{jh}(-bac)\in E(n,BAC),\\
&[y_{ij}(a,b),t_{hi}(c)]=t_{hi}(-cab)t_{hj}(caba)\in E(n,CAB),\\
&[y_{ij}(a,b),t_{hj}(c)]=t_{hi}(-cbab)t_{hj}(cba+cbaba)\in E(n,CBA).
\end{align*}
\par\noindent
All factors on the right hand side belong already to 
$E(n,(A\circ B)\circ C)$.
\par\smallskip
$\bullet$ Finally, for $(k,l)=(i,j),(j,i)$ we can take an $h\neq i,j$
and rewrite $t_{kl}(c)$ as a commutator 
$t_{ij}(c)=[t_{ih}(c),t_{hj}(1)]$ or $t_{ji}(c)=[t_{jh}(c),t_{hi}(1)]$.
Let us consider the first case, the second one is similar. One has
$$ z=[y_{ij}(a,b),t_{ij}(c)]=\big[y_{ij}(a,b),[t_{ih}(c),t_{hj}(1)]\big]=
{}^{y_{ij}(a,b)}[t_{ih}(c),t_{hj}(1)]\cdot t_{ij}(-c), $$
\noindent
shifting conjugation inside the commutator, we get
$$ z=\big[{}^{y_{ij}(a,b)}t_{ih}(c),{}^{y_{ij}(a,b)}t_{hj}(1)\big]
\cdot t_{ij}(-c) = [t_{ih}(c)u,t_{hj}(1)v]\cdot t_{ij}(-c), $$
\noindent
where, by the above (the second case corresponds to $c=1$),
 we have
 \begin{align*}
&u=t_{ih}(abc+ababc)t_{jh}(babc)\in E(n,ABC),\\
&v=t_{hi}(-bab)t_{hj}(ba+baba)\in E(n,BA). 
\end{align*}
\noindent
Clearly,
$$ u\in E(n,ABC)\le E(n,(A\circ B)\circ C)\le
\big[E(n,A\circ B), E(n,C)\big], $$
\noindent
and stays there under all elementary conjugations. Thus, 
$$ z\equiv [t_{ih}(c),t_{hj}(1)v]\cdot t_{ij}(-c)
\pamod{E(n,(A\circ B)\circ C)}. $$
\noindent
Using multiplicatibvity of the commutator on the right hand side
with respect to the second factor, we get
$$ z\equiv [t_{ih}(c),t_{hj}(1)]\cdot
{}^{t_{hj}(1)}[t_{ih}(c),v]\cdot t_{ij}(-c)
\pamod{E(n,(A\circ B)\circ C)}. $$
\noindent
However, the right hand side of the above congruence equals
$$ {}^{ t_{ij}(c)t_{hj}(1)}[t_{ih}(c),v]\in 
[E(n,R,A\circ B),E(n,C)], $$
\noindent
so that $z\in [E(n,R,A\circ B),E(n,C)]$, as claimed.

\par\smallskip
For $n\ge 4$ Lemma 8 already follows from Lemma 11 and
Lemma 7. 
\par\smallskip
From the previous lemma we already know that 
$$ \big[E(n,A\circ B),\big[E(n,C),E(n,D)\big]\big]=
\big[E(n,A\circ B),E(n,C\circ D)\big], $$
\noindent
and that
$$ \big[\big[E(n,A),E(n,B)\big],E(n,C\circ D)\big]=
\big[E(n,A\circ B),E(n,C\circ D)\big]. $$
\par
Thus, it only remains to prove that 
$$ [y_{ij}(a,b),y_{hk}(c,d)]\in \big[E(n,A\circ B),E(n,C\circ D)\big], $$
\noindent
for $1\le i\neq j\le n$, $1\le h\neq k\le n$, $a\in A$, 
$b\in B$, $c\in C$, $d\in D$. Conjugations by elements
$x\in E(n,R)$ do not matter, since they amount to extra
factors from the above triple commutators, which are
already accounted for.
\par
Now, for $n\ge 4$ this already finishes the proof, since in
this case by Lemma 11 we can move $y_{hk}(c,d)$ modulo 
$E(n,R,C\circ D)$ to a position, where it commutes with 
$y_{ij}(a,b)$. 
\par\smallskip
To extend this lemma also to the case $n=3$
one would need to perform calculations similar to, but much
fancier than the last item in the above proof of Lemma 7.
One starts with expanding the commutator $[y_{ij}(a,b),y_{jh}(c,d)]$
as above, but then it is not immediate to verify that the factors 
belong where they should.


\section{Final remarks}

Of course, the first question that immediately occurs is 
whether Theorem 1 holds also for $n=3$. We are slightly 
more inclined to believe in the positive answer.

\begin{problem}
Prove that Lemma $8$ and Theorem $1$ hold also for $n=3$.
\end{problem}

However, a negative solution of this problem would be also
very interesting and not easy to obtain, since a counter-example
should be higher-dimensional and very non-commutative. 
On the other hand, since the quotient $[E(n,I),E(n,I)]/E(n,R,I^2)$
is closely related to relative $K_2$, the rank two groups may be problematic here. Thus, even in the commutative case Matthias
Wendt constructed counter-examples to the centrality of 
$\St(3,R)\map E(3,R)$, see \cite{Wendt}.

\par\smallskip
One may ask, whether double commutators $[E(n,A),E(n,B)]$
can be reduced further, in other words, are themselves always
equal to $E(n,R,A\circ B)$. It is classically known this 
is not the case, even when $R$ is commutative, and even in the
stable range. Below we discuss some counter-examples.
\par
However, this is indeed the case in the important instance when
$A$ and $B$ are comaximal. Let us record another amusing 
corollary of \cite{NZ2}, Theorem~1. For quasi-finite rings this
result was established in \cite{Vavilov_Stepanov_revisited},
Theorem 5 and \cite{Hazrat_Vavilov_Zhang}, Theorem 2A, 
but for arbitrary associative rings it is new.

\begin{theorem}
Let $R$ be any associative ring with $1$, let $n\ge 3$, and let 
$A$ and $B$ be two-sided ideals of $R$. If $A$ and $B$ are
comaximal, $A+B=R$, then
$$ [E(n,A),E(n,B)]=E(n,R,A\circ B). $$
\end{theorem}
\begin{proof}
Since $A$ and $B$ are comaximal, there exist $a'\in A$ and
$b'\in B$ such that $a'+b'=1\in R$. But then by Lemmas 10 and 12
one has
$$ y_{ij}(a,b)=y_{ij}(a(a'+b'),b)\equiv y_{ij}(aa',b)\cdot y_{ij}(ab',b)\equiv e\pamod{E(n,R,A\circ B)}. $$
\end{proof}

Lemma 12 implicates that the difference between the mutual 
commutator
subgroup $[E(n,A),E(n,B)]$ and the relative elementary subgroup $E(n,R,A\circ B)$ resides in $A/A^2$ and $B/B^2$. Thus, it is no
surprise that the simplest counter-examples should be of the form 
$[E(n,I),E(n,I)]>E(n,R,I^2)$. These examples show, in particular,
that the second type of generators in Lemma~4 are necessary,
and that in general double commutators cannot be further
reduced.
\par\smallskip
$\bullet$ The first counter-example was constructed by Alec
Mason and Wilson Stothers \cite{Mason_Stothers, Mason_commutators_2}. In their example $R=\Int[i]$ is the
ring of Gaussian integers, whereas $I$ is its prime ideal
$\prim=(1+i)R$. In general, when $I$ is an ideal of a Dedekind 
ring of arithmetic type, an explicit formula for $\SK_1(R,I)$ was 
obtained by Hyman Bass, John Milnor and Jean-Pierre Serre \cite{Bass_Milnor_Serre}. Now a straightforward calculation using 
this formula shows that
$$ E(n,\Int[i],\prim^6)<[E(n,\Int[i],\prim^3),E(n,\Int[i],\prim^3)]]<\SL(n,\Int[i],\prim^6), $$
\noindent
where {\it both\/} indices are equal to 2.
\par\smallskip
However, there are many further counter-examples which show 
that this phenomenon is of a very general nature. Several such 
counter-examples are discussed in the works of Suzan Geller and Charles Weibel, see, for instance, \cite{Geller_Weibel}, 
Example 6.1 and Example 7.  In fact these examples show that
even $E(R,I^2)<[E(I),E(I)]$. But by \cite{NZ2}, Theorem 5, 
the stability map
\begin{multline*}
\big[E(n,A),E(n,B)\big]/E(n,R,AB+BA) \map\\
\big[E(n+1,A),E(n+1,B)\big]/E(n+1,R,AB+BA) 
\end{multline*}
is surjective for all for all $n\ge 3$ and it is even an isomorphism
when $n\ge\max(\sr(A)+1,3)$, by \cite{Hazrat_Vavilov_Zhang},
Lemma 15. Thus, these counter-examples work already at
the level of $\GL(3,R)$ or $\GL(4,R)$.
\par\smallskip
$\bullet$ Let $R=\Rat[x,y]$, $I=xR+yR$. Then
$$ z=\begin{pmatrix} 
1-xy&x^2&0\\ -y^2&1+xy&0\\ 0&0&1
\end{pmatrix} = \left[
\begin{pmatrix} 
1&0&x\\ 0&1&y\\ 0&0&1
\end{pmatrix},
\begin{pmatrix} 
1&0&0\\ 0&1&0\\ -y&x&1
\end{pmatrix} 
\right]\in [E(3,I),E(3,I)]. $$
\noindent
But by \cite{Weibel} one has $K_1(R,I^2)=\GL(R,I^2)/E(R,I^2)\cong\Rat$ and under
this isomorphism the Mennicke symbol 
$\displaystyle\left[\!{x^2\atop 1-xy}\!\right]$ goes to $2\in\Rat$.
This means that $z\notin E(R,I^2)$ and thus $z\notin E(3,R,I^2)$.
\par\smallskip
$\bullet$ Let $R=\Int[x]$, $I=xR$. Then clearly
$$ y_{21}(x,x)=
\begin{pmatrix} 
1-x^2&x^3&0\\ -x^3&1+x^2+x^4&0\\ 0&0&1
\end{pmatrix} \in [E(3,I),E(3,I)]. $$
\noindent
But 
the Mennicke symbol
$\displaystyle\left[\!{x^3\atop 1-x^2}\!\right]$ is non-trivial,
so that $y_{21}(x,x)\notin E(R,I^2)$ and thus 
$y_{21}(x,x)\notin E(3,R,I^2)$.
\par\smallskip

Let us reiterate \cite{NZ2}, Problem 1, and make it more specific.
We believe that most of the relations are already listed above
in Lemmas 10 and 11.

\begin{problem}
Give a presentation of
$$ \big[E(n,A),E(n,B)\big]/\EE(n,R,AB+BA) $$
\noindent
by generators and relations. Does this presentation depend on 
$n\ge 3$? 
\end{problem}

If it does not, this would in particular produce a solution of
Problem~1 on a completely different track.
\par
For Chevalley groups the ground ring is commutative from
the outset, so that the standard commutator formula holds.
But for non-algebraic forms of classical groups the alternative
approach taken in the present paper very much makes sense. 
In fact, for
Bak's unitary groups it is now fully incorporated in our paper
\cite{NZ3}, where we, in particular reduce the generating 
sets of the elementary commutators 
$[\EU(2n,A,\Gamma),\EU(2n,B,\Delta)]$ and remove all 
remaining restrictions on the form ring $(R,\Lambda)$ in
all results if \cite{Hazrat_Vavilov_Zhang, RNZ5}, and other
related works, pertaining to the elementary unitary groups.
In \cite{NZ3} they are established for {\it arbitrary\/} form
rings, whereas before they were only known for form rings 
subject to some commutativity conditions such as quasi-finiteness.
\par\smallskip
We are very grateful to  Roozbeh Hazrat and Alexei Stepanov 
for ongoing discussion of this circle of ideas, and long-standing cooperation over the last decades. 


\end{document}